\documentclass{amsart}
\usepackage{amssymb}
\usepackage{a4wide}
\usepackage[all]{xypic}

\DeclareMathOperator{\Aut}{Aut}
\DeclareMathOperator{\GL}{GL}

\newcommand{\Z}{{\mathbb Z}}
\newcommand{\Q}{{\mathbb Q}}

\newcommand{\C}{{\mathbb C}}

\newtheorem{theorem}{Theorem}[section]
\newtheorem{proposition}[theorem]{Proposition}
\newtheorem{lemma}[theorem]{Lemma}

\newtheorem{corollary}[theorem]{Corollay}
\newtheorem*{remark}{Remark}

\begin{document}

\title{The $R_\infty$ property for free groups, free nilpotent groups and free solvable groups}
\author[K.\ Dekimpe]{Karel Dekimpe}
\author[D.\ Gon\c{c}alves]{Daciberg Gon\c{c}alves}
\address{KULeuven Kulak\\
E. Sabbelaan 53\\
8500 Kortrijk\\
Belgium}
\email{karel.dekimpe@kuleuven-kulak.be}
\address{Departamento de Matem\'atica--IME--USP\\ 
Universidade de S\~{a}o Paulo\\ 
S\~{a}o Paulo\\
Brasil}
\email{dlgoncal@ime.usp.br}

\begin{abstract}

Let $F$ be either a free nilpotent group of a given class and of finite rank
or a free solvable group of a certain derived length and of finite rank. 
We show precisely which ones have the $R_{\infty}$ property. 
Finally,  we also show that the free group 
of infinite rank does not have the $R_{\infty}$ property. 
\end{abstract}

\maketitle
\section{Introduction}

Let $G$ be a group and $\varphi$ be an endomorphism of $G$. Then two elements 
$x,y$ of $G$ are said to be Reidemeister equivalent, if there exists a third element 
$z\in G$ such that $x = z y \varphi(z)^{-1}$. The equivalence classes  are called the Reidemeister classes or twisted conjugacy classes. The Reidemeister number of $\varphi$, which we will denote by $R(\varphi)$, is the number of those Reidemeister classes of $\varphi$. This number is either a positive integer or $\infty$. 

\medskip

These Reidemeister numbers are a relevant ingredient for studying the fixed point properties of the homotopy class of a self map on a topological space. In this situation, the group $G$ will be the fundamental group $\pi_1(X)$ of the space and the 
morphism $\varphi=f_\sharp$ is the one which is induced by the map $f$. Under certains hypothesis the 
Reidemeister number $R(\varphi)$ is then exactly the number of essential  fixed point 
classes of $f$ if $R(\varphi)$ is finite and the number of essential  fixed point 
classes of $f$ is zero if  $R(\varphi)$ is infinite. See \cite{j} and \cite{w} and the references 
therein for more information.

\medskip

A group $G$ has the   
$R_\infty$ property if for every automorphism $\varphi$ of $G$ the Reidemeister number is 
infinite. In recent years many works have studied the question of 
which groups $G$ have the 
$R_\infty$ property.  We refer to \cite{FLT} for an overview of the results which have been obtained in this direction. 
The present work will also give a contribution for this problem,  where 
we will consider groups $G$ from the families  of free groups, free nilpotent groups and
free solvable groups, of countable rank (finite and infinite). 
The results in our paper are closely related with those in   \cite{GW} and \cite{Ro} 

\bigskip

In this paper we will use $F_r$ to denote the free group on $r>1$ generators $x_1, x_2, \ldots, x_r$. 
We say that $F_r$ is the free group of rank $r$.
Likewise, $F_\infty$ will be used to refer to the free group on a countable number of generators
$x_1,x_2,x_3,\ldots$. When $G$ is a group, we will use $\gamma_i(G)$ to denote the terms of its lower
central series, i.e.
\[ \gamma_1(G)=G,\; \gamma_{i+1}(G)  = [\gamma_i(G), G] \mbox{ for $i>0$}.\]
A group is said to be nilpotent of class $\leq c$ when $\gamma_{c+1}(G)=1$.
The group 
\[ N_{r,c}= F_r/\gamma_{c+1}(F_r) \]
is the free nilpotent group of class $c$ and rank $r$, while the group 
\[ N_{\infty,c}=   F_\infty/\gamma_{c+1}(F_\infty) \]
is free nilpotent of class $c$ and countable rank.

The derived series of a group $G$ is also defined recursively by 
\[ G^{(0)}= G, \; G^{(1)}= [ G,G], \mbox{ and } G^{(i+1)}=  [G^{(i)}, G^{(i)}], \mbox{ for $i>0$}.\]
A group is said to be solvable with derived length $\leq k$ when $G^{(k)}=1$.  A solvable group with derived length $\leq 2 $ is also said to be metabelian.

The group 
\[ S_{r,k}= F_r/F_r^{(k)}\]
is the free solvable group of rank $r$ and derived length $k$,  while the group 
\[ S_{\infty,k}=   F_\infty/F_{\infty}^{(k)}\]
is the  free solvable group of countable rank and of derived length $k$. 

Finally, we will also use 
\[ M_{r,c}= \frac{F_r}{\gamma_{c+1}(F_r) F_r^{(2)}},\] which is the free $c$-step nilpotent and metabelian group  of rank $r$, and  
\[ M_{\infty,c}= \frac{F_{\infty}}{\gamma_{c+1}(F_{\infty}) F_{\infty}^{(2)}},\]  which is the 
free $c$-step nilpotent and metabelian group  of    countable  rank. 

 Note that there 
is a natural projection $N_{r,c}\rightarrow M_{r,c}$ and any automorphism of $N_{r,c}$ induces an automorphism of $M_{r,c}$. 


In this paper we will completely classify which of the above introduced groups have the $R_\infty$ property. Notoriously, we show that $F_\infty$, the free group of infinite rank, does not have property $R_\infty$.
\section{The $R_\infty$ property for free nilpotent groups of finite rank}

In \cite{Ro}, it was shown that all free nilpotent groups $N_{r,c}$ have the $R_\infty$ property provided 
$c\geq 2r $ (with the exception of $r=3$, where it was required that $c\geq 12$). We will provide a new proof of this 
result and moreover, we will show that this result is sharp for $r\neq3$ and we also provide the sharp result for $r=3$.
I.e.\ we will prove the following result in {\bf Theorem~\ref{free-nilpotent-finite-rank}}:

\medskip

{\bf Theorem:}
Let $r>1$ and $c\geq 1$ be positive integers.\\
Then $N_{r,c}$ has the $R_\infty$--property if and only if $c\geq 2r$.

\medskip

Before we are able to prove this theorem, we need to recall how to associate a Lie ring to a given group.

It is well known that to any group $G$, one can associate a graded Lie ring $L(G)= \displaystyle \bigoplus_{i=1}^\infty L_i(G)$ over $\Z$, where 
\[ L_i(G)= \gamma_i(G)/\gamma_{i+1}(G)\]  
and the Lie bracket $[\cdot , \cdot]_L$ is completely determined by 
\[ \forall x \in \gamma_{i}(G), y \in \gamma_{j}(G): \; [ x+ \gamma_{i+1}(G), y+\gamma_{j+1}(G)]_L = [x,y] + \gamma_{i+j+1} (G) \in L_{i+j}(G).\]
In this expression $[x,y]$ denotes the group commutator in $G$.

Any automorphism $\varphi\in \Aut(G)$ determines an automorphism $\varphi_i \in \Aut(\gamma_i(G)/\gamma_{i+1}(G))$ and moreover, the 
linear map of $L(G)$ which restricted to $L_i(G)$ equals $\varphi_i$ is a (graded) Lie ring automorphism.

\medskip

Any endomorpism of a finitely generated abelian group $A$ maps the torsion subgroup 
$\tau(A)$ of $A$ to itself and hence this endomorphism induces an endomorphism 
of the free abelian group $A/\tau(A)$. When we talk about the eigenvalues of such a morphism, we will mean the eigenvalues of the morphism  of $A/\tau(A)\cong \Z^d$,
which can be seen as a $d\times d$--matrix with integer entries.  
 The following result is known and is essentially  Corollary 4.2 of \cite{Ro}. 

\medskip

\begin{lemma}\label{critnilp}
Let $G$ be a finitely generated nilpotent group. For any
automorphism $\varphi$ of $G$ we let $\varphi_i\in \Aut(\gamma_i(G)/\gamma_{i+1}(G))$ denote the induced automorphism. Then we have  that 
\[ R(\varphi)=\infty \Leftrightarrow \exists  i  \mbox{ such that $\varphi_i$ has $1$ as an eigenvalue}.\]
\end{lemma}

\begin{remark}
As $G$ is finitely generated, then also all $\gamma_i(G)/\gamma_{i+1}(G)$ are finitely generated, and hence we can talk about the eigenvalues of $\varphi_i$.
In \cite{Ro} it was assumed that all $\gamma_i(G)/\gamma_{i+1}(G)$ are torsion free. However, there is no difficulty to generalize his result to the general case above.
\end{remark}

Now, for any group $G$, the automorphisms $\varphi_i$ ($i>1$) are completely determined by the map $\varphi_1$. It is then also no 
surprise that for the nilpotent groups $N_{r,c}$ and $M_{r,c}$, we can express the eigenvalues of $\varphi_i$ in terms of the eigenvalues of $\varphi_1$. This is the content of the following lemma.

\begin{lemma}\label{eigenvaluesquotients} Let $G= N_{r,c}$ or $G=M_{r,c}$.
Let $\varphi$ be an automorphism of $G$ and let $\varphi_i$ denote the induced automorphism 
of $\gamma_{i}(G)/\gamma_{i+1}(G)$. Let 
\[ \lambda_1, \lambda_2, \ldots, \lambda_r \in \C\] 
denote the eigenvalues of $\varphi_1$ (listed as many times as their multiplicity). Then any eigenvalue $\mu$ of $\varphi_i$ 
(with $2\leq i \leq c$) can be written as an $i$--fold  product 
\[ \mu =\lambda_{j_1} \lambda_{j_2} \cdots \lambda_{j_i}.\]
Moreover, any such a product in which not all of the $j_i$ are equal does appear as an eigenvalue of $\varphi_i$.
\end{lemma}
\begin{proof}
We present the proof in case $G=M_{r,c}$. The case where $G=N_{r,c}$ is analogous.  

For $G=M_{r,c}$, the Lie ring $L(G)$ is the free $c$-step nilpotent and metabelian Lie ring over $\Z$.

Let $\varphi_\ast$ denote the Lie ring automorphism of $L(G)$, which is induced by $\varphi$.
It is obvious that the eigenvalues of $\varphi_i$ are exactly the same as the eigenvalues of 
$\varphi_\ast$ restricted to $L_i(G)=\gamma_i(G)/\gamma_{i+1}(G)$.

In order to easily compute these eigenvalues, we will consider the complexification of $L(G)$:
\[ L^\C(G)=L(G)\otimes \C = \sum_{i=1}^{\infty} L_i(G)\otimes \C. \]
 
Then $\varphi_\ast$ can also be seen as an automorphism of the Lie algebra $L^\C(G)$. As 
$\Aut(L^\C(G))$ is an algebraic group, we know that also the semi-simple part $\varphi^s_\ast$ of 
$\varphi_\ast$ is an automorphism of $L^\C(G)$. The automorphism $\varphi^s_\ast$ also preserves the grading of $L(G)$ and moreover, for any $i=1,2,\ldots$ the restrictions of 
$\varphi_\ast$ and $\varphi^s_\ast$ to $L_i^\C(G)$  have the same eigenvalues. Hence, we can assume that $\varphi_\ast$ itself is semi-simple. 

Let us fix a basis $X_1,\; X_2,\ldots, X_r$ of $L_1^\C(G)$ consisting of eigenvectors 
for $\varphi_\ast$ where $X_j$ is an eigenvector for the eigenvalue $\lambda_j$.

Now, it is a well known result due to Chen (\cite{C}) that $L_i(G)$ has a basis consisting of the vectors
\[ [\ldots [[[x_{j_1}, x_{j_2}]_L,x_{j_3}]_L,x_{j_4}]_L,\ldots, x_{j_i}]_L\mbox{ \ with \ }
j_1 > j_2 \leq j_3\leq j_4 \leq \cdots \leq j_i.\]

It is now obvious that these vectors are also eigenvectors for $\varphi_\ast$ with respect to the eigenvalues 
\[ \lambda_{j_1}\lambda_{j_2}\lambda_{j_3} \lambda_{j_4} \cdots \lambda{j_i} .\]

This proves the lemma.
\end{proof}

Having the lemmas above in mind, it should not come as a surprise that the following proposition will be very useful in proving our results on the $R_\infty$--property for free nilpotent groups of
finite rank.

\begin{proposition}\label{keyprop}
For every positive integer $r$ there exists a matrix $A_r\in \GL(r,\Z)$ with 
$r$ distinct eigenvalues $\lambda_1,\lambda_2, \ldots ,\lambda_r$  and such that 
\[ \forall k \in \{1,2,\ldots,2r-1\},\;
\forall i_1, i_2, \ldots, i_k \in \{1,2,\ldots, r\}: \;
\lambda_{i_1} \lambda_{i_2} \cdots \lambda_{i_k} \neq 1.\]
\end{proposition}
Remark: this is really the best possible result, because for $k=2r$, we have that the following product 
\[ \lambda_1\lambda_1 \lambda_2\lambda_2 \ldots \lambda_r \lambda_r = \det(A_r)^2= 1.\]
Note that we must also have that $A_r$ has determinant $-1$.

\medskip

The idea for proving this lemma is to make use of so called Pisot numbers (also called Pisot--Vijayaraghavan numbers). 

A Pisot number is a real algebraic integer $\theta$ greater than 1, such that all conjugates of $\theta$ lie strictly inside the unit disk (so are of modulus $<1$). Actually,
for us, $\theta$ does not have to be positive, we only need $|\theta|>1$.

\medskip

In the following lemma, we show how one can construct such numbers of any degree and prove some interesting properties about them.

\begin{lemma}\label{roots}
Let $n>1$ be an integer and let 
$p(x)= x^n +a_{n-1}x^{n-1} + \cdots + a_2 x^2 + a_1  x + (-1)^{n+1}$, for some 
integers $a_i$ with 
$|a_{n-1}|> |a_{n-2}| + \cdots + |a_2| + |a_1| + 2$

Then 
\begin{enumerate}
\item $p(x)$ has one real root $\theta_1$ with $|\theta_1|>1$ and all other roots lie strictly inside the unit circle.
\item $p(x)$ is an irreducible polynomial over $\Q$ and hence $p(x)$ is the minimal polynomial of $\theta_1$.
\item If $\theta_2, \theta_3, \ldots, \theta_n\in \C$ denote the other roots of $p(x)$,
then $\theta_1\theta_2\ldots \theta_n=-1$.
\item If for some $d_1,d_2,\ldots,d_n\in \Z$ we have that $\theta_1^{d_1} \theta_2^{d_2} \ldots \theta_n^{d_n} =1$, then there
exists an integer $z\in \Z$ such that $d_1=d_2=\cdots = d_n=2 z$.
\end{enumerate}
\end{lemma}

\begin{proof} $ $
\begin{enumerate}
\item This is a well known consequence  of Rouch\'e's Theorem. E.g.\ see theorem 2.2.5 and its proof on page 56 of 
\cite{Pr}.
\item Suppose that $p(x)=q(x) r(x)$, where $q(x)$ and $r(x)$ are monic polynomials, then 
both $q(x), r(x) \in \Z[x]$. But one of these polynomials, say $q(x)$,  only has (non zero) roots of modulus
$<1$. But then the constant term of $q(x)$ is, up to sign, the product of these roots and hence is a non-zero 
integer of modulus $<1$, which is a contradiction. It follows that $p(x)$ is the minimal polynomial of $\theta_1$.
\item Note that the constant term of $p(x)= (-1)^n \theta_1\theta_2\ldots \theta_n$, showing that 
$\theta_1\theta_2\ldots \theta_n=-1$
\item This part is essentially Proposition 3.1 of \cite{BH}. For the reader's convenience, we recall the argument here.\\
Suppose that 
\[ \theta_1^{d_1} \theta_2^{d_2} \ldots \theta_n^{d_n} =1 \]
Now, let $m=\min \{d_1,d_2,\ldots d_n\}$, then 
\[ \theta_1^{d_1-m} \theta_2^{d_2-m} \ldots \theta_n^{d_n-m} =(\theta_1 \theta_2 \ldots \theta_n )^{-m}= (-1)^m.\]
Note that $d_i-m \geq 0$ for all $i$ and that there exists a $k$ such that $d_k-m=0$.
As the Galois group of $\Q(\theta_1,\theta_2,\ldots,\theta_n)$ over $\Q$ acts transitively on 
the set $\{ \theta_1,\theta_2, \ldots,\theta_n\}$, there exists a $\sigma$ in this Galois group mapping $\theta_k$ 
onto $\theta_1$. By applying $\sigma$ to both sides of the equation above, we may assume that $d_1-m=0$. But then 
\[ \theta_2^{d_2-m} \ldots \theta_n^{d_n-m}= (-1)^m\]
where the left hand side consists of a product of powers of elements $\theta_i$ of modulus $<1$. This equation can only be satisfied
when all $d_i-m=0$ and $m$ is even. Hence we have shown that $\exists z \in \Z$ such that 
\[ d_1=d_2=\cdots =d_n=2 z. \]
\end{enumerate}
\end{proof}

\medskip

We are now ready for:
\begin{proof}[Proof of Proposition \ref{keyprop}]
For $r=1$, we consider the $1\times 1$--matrix $(-1)$.\\
For $r>1$ we let $A_r$ be the companion matrix of the polynomial $p(x)=x^r-3x^{r-1} +(-1)^{r+1}$. Then, the characteristic polynomial 
of $A_r$ is exactly $p(x)$, which satisfies the criteria of lemma~\ref{roots}, and the last part of lemma above shows that 
the eigenvalues of $A_r$, which are the roots of $p(x)$, satisfy the condition claimed in the proposition. 

\end{proof}

\begin{theorem}\label{free-nilpotent-finite-rank}
Let $r,c$ be positive integers, with $r>1$. Let $G=N_{c,r}$ or $G=M_{c,r}$.\\
Then $G$ has property $R_\infty$ $\Leftrightarrow$ $c \geq 2 r$
\end{theorem}
\begin{proof}
Let $\varphi$ be an automorphism of $G$ and let  
$\varphi_1$ be the induced automorphism on $G/\gamma_2(G)\cong \Z^r$ and denote the eigenvalues of $\varphi_1$ 
by $\lambda_1, \lambda_2, \ldots , \lambda_r$.  

\medskip

By lemma~\ref{eigenvaluesquotients} we know that the eigenvalues of the induced 
automorphism $\varphi_i$ on the  quotient 
$\gamma_i(G)/\gamma_{i+1}(G)$ consist of products of the form 
\[ \lambda_{j_1} \lambda_{j_2} \ldots \lambda_{j_i}.\]
Now, if $c \geq 2r$, we can consider the quotient $\gamma_{2r}(G)/\gamma_{2r+1}(G)$, on which 
the induced automorphism has 
\[ (\lambda_1\lambda_2\cdots \lambda_r)^2=\det(A)^2=1 \]
as one of its eigenvalues. Hence $R(\varphi)=\infty$, by lemma~\ref{critnilp}.

\medskip

On the other hand, if $c<2r$, then consider an automorphism $\varphi$ of $G$ inducing the automorphism $A_r$ from proposition~\ref{keyprop} on  $G/\gamma_2(G)\cong \Z^r$. This is possible since it is known that 
the natural map 
\[\Aut(F_{r}) \rightarrow \Aut \left( \frac{F_{r}}{\gamma_2(F_{r})}\right)=\GL(r,\Z)\]
is onto. From this, it easily follows that also the natural map 
\[ \Aut(G)\rightarrow \Aut(G/\gamma_2(G))=\GL(r,\Z)\]
is onto, see \cite{mks} Theorem $N4$, section 3.5.


Then the properties of $A_r$ given in proposition~\ref{keyprop}, together with lemma~\ref{critnilp} imply that $R(\varphi)<\infty$, showing that $G$ does not have property $R_\infty$.

\end{proof}

\section{Free soluble groups.}

Now we consider the free solvable groups  of rank $r>1$ and derived length $k$ i.e.\ $S_{r,k}=F_r/F_r^{(k)}$.

\begin{theorem}  Let $r>1$ be an integer. The group    $S_{r,1}=F_r/[F_r, F_r]=F_r/F_r^{(1)}\cong\Z^r$ does not have the $R_{\infty}$ property. On the other hand 
the groups $S_{r,k}=F_r/F_r^{(k)}$ have property  $R_{\infty}$ for all $k\geq 2$.
\end{theorem}

\begin{proof} 
The result for $k=1$ is well known, so let $k>1$. 

As $S_{r,2}=\displaystyle \frac{S_{r,k}}{S_{r,k}^{(2)}}$ is a quotient of  $S_{r,k}$ by a characteristic subgroup, it is enough to prove that $S_{r,2}$ has property $R_\infty$. Now, we choose $c\geq 2 r$ and consider the quotient
\[\frac{S_{r,2}}{\gamma_{c+1}(S_{r,2})}= 
\frac{F_r}{\gamma_{c+1}(F_r) F_r^{(2)}}=M_{r,c}.\]
By theorem~\ref{free-nilpotent-finite-rank} we know that $M_{r,c}$ has property 
$R_\infty$ and hence so does $S_{r,2}$ and all groups $S_{r,k}$ for $k\geq 2$. 
 \end{proof}

\section{Free nilpotent groups with a countable number of generators}
In this section we will construct a very explicit example of an automorphism $\varphi_c$ 
of $N_{\infty,c}$ (for any $c\geq 1$), such that $R(\varphi_c)=1$. It follows that $N_{\infty,c}$ does not have property $R_\infty$. This last fact will also follow from the next section, but the results in that section are far from being so explicit as the ones in this section.

\medskip

So, we will consider $F_\infty$, the free group on an infinite but still countable number of generators $x_0,\,x_1,\,x_2,\,x_3,\ldots$.
Let $\varphi\in \Aut(F_\infty)$ be the automorphism determined  by 
\[ \varphi(x_0)= x_0,\; \forall i>0:\varphi(x_i)=x_{i-1} x_i\]
and let $\varphi_c\in \Aut(N_{\infty,c})$ be the induced automorphism on $N_{\infty,c}$.
(It is easy to see that $\varphi$ does determine an automorphism of $F_\infty$ and hence also induces one on $N_{\infty,c}$).
\begin{lemma}
Let $n$ be a postive integer, then 
\[ \forall x \in \gamma_n(F_\infty),\; \exists y \in \gamma_n(F_\infty):\; x\gamma_{n+1}(F_\infty)= y \varphi(y)^{-1} \gamma_{n+1}(F_\infty).\]
\end{lemma}
\begin{proof}
We use induction on $n$.\\ 
Let $n=1$, then for any generator $x_i$, we have that 
\[ x_i\gamma_2(F_\infty) = x_{i+1}^{-1}\varphi(x_{i+1}^{-1})^{-1} \gamma_2(F_\infty).\]
As $F_\infty/\gamma_2(F_\infty)$ is a free abelian group on the generators $x_i\gamma_2(F_\infty)$, we can conclude that  for all
$x\gamma_2(F_\infty)$ there exists a $y \gamma_2(F_\infty)$ such that $x\gamma_2(F_\infty) = y \varphi(y)^{-1}\gamma_2(F_\infty)$, which implies that the lemma
holds for $n=1$.

Now assume that $n>1$ and the lemma holds up to the case $n-1$.
The group $\gamma_{n}(F_\infty)/\gamma_{n+1}(F_\infty)$ is generated by the cosets of elements of the form $[x_i,y]$ where $i\geq 0$ and $y\in \gamma_{n-1}(F_\infty)$. 

\medskip

To prove the lemma, it is enough to show that any coset $[x_i,y]\gamma_{n+1}(F_\infty)$ of such an element is the product of 
cosets of the form $z\varphi(z)^{-1}\gamma_{n+1}(F_\infty)$ for some $z\in \gamma_n(F_\infty)$ (use that $\gamma_n(F_\infty)/\gamma_{n+1}(F_\infty)$ is central 
in $F_\infty/\gamma_{n+1}(F_\infty)$).

We will prove this last claim by induction on $i$. 
Let $i=0$. As we assume that the lemma is valid up to $n-1$, there exists a $w\in \gamma_{n-1}(F_\infty)$ such that 
\begin{equation}\label{fact}
 y \gamma_n(F_\infty) = w \varphi(w)^{-1} \gamma_n(F_\infty).
\end{equation}
Using this, we find
\begin{eqnarray*}
[x_0, y] \gamma_{n+1}(F_\infty) & = & [x_0 , w \varphi(w)^{-1} ]  \gamma_{n+1}(F_\infty)\\
                         & = & [x_0 , w ] [x_0, \varphi(w)^{-1} ]\gamma_{n+1}(F_\infty)\\
                         & = & [x_0 ,w] [x_0, \varphi(w) ]^{-1} \gamma_{n+1}(F_\infty)\\
                         & = & [x_0, w] (\varphi[x_0,w])^{-1}\gamma_{n+1}(F_\infty)\\
                         & = & z \varphi(z)^{-1}\gamma_{n+1}(F_\infty) \mbox{ for }z=[x_0,w]\in \gamma_n(F_\infty).
\end{eqnarray*}
Now, let $i>0$ and assume we already proved our claim for all generators $x_0, x_1, \ldots, x_{i-1}$.
Again, we use the relation (\ref{fact}) in the following computation:
\begin{eqnarray*}
[x_i,y] \gamma_{n+1}(F_\infty) & = & [x_i , w \varphi(w)^{-1} ]  \gamma_{n+1}(F_\infty)\\
                        & = & [x_i, w] [x_i, \varphi(w)^{-1} ]  \gamma_{n+1}(F_\infty)\\
                        & = & [x_i, w] [ x_{i-1}^{-1}\varphi(x_i),  \varphi(w)^{-1} ]  \gamma_{n+1}(F_\infty)\\
                        & = & [x_i, w] [ x_{i-1}^{-1},  \varphi(w)^{-1} ] [ \varphi(x_i),  \varphi(w)^{-1} ]  \gamma_{n+1}(F_\infty)\\
                        & = & [x_{i-1}, \varphi(w)] [x_i, w] \varphi([x_i,w])^{-1}  \gamma_{n+1}(F_\infty)
\end{eqnarray*}
which by the induction hypothesis on $i$ is a product of elements of the form $z\varphi(z)^{-1} \gamma_{n+1}(F_\infty)$ with 
$z\in\gamma_n(F_\infty)$. This concludes the proof of the lemma.
\end{proof}
\bigskip

\begin{proposition}
Let $c>0$ be a positive integer and let $\varphi_c$ be the automorphism of $N_{\infty,c}=F_\infty/\gamma_{c+1}(F_\infty)$ which is induced by $\varphi$.
Then $R(\varphi_c)=1$. In particular, the free nilpotent group $N_{\infty,c}$ of class $c$ on an infinite countable
number of generators does not have the $R_\infty$ property.
\end{proposition}

\begin{proof} Again we proceed by induction. The case $c=1$ is clear, because the previous
lemma shows that the map Id$-\varphi_1$ is a surjective map on $N_{\infty,1}=F_\infty/\gamma_2(F_\infty)$. 

Now, let $c>1$ and assume that the proposition holds up to $c-1$. We have the following commutative diagram, where the 
horizontal rows are central group extensions
\[\xymatrix{1 \ar[r] & \gamma_c(F_\infty)/\gamma_{c+1}(F_\infty) \ar[r]\ar[d]_{\varphi'}  & 
                          F_\infty /\gamma_{c+1}(F_\infty) \ar[r]\ar[d]_{\varphi_c} & 
                               F_\infty/\gamma_{c}(F_\infty) \ar[r] \ar[d]_{\varphi_{c-1}}& 1 \\
            1 \ar[r] & \gamma_c(F_\infty)/\gamma_{c+1}(F_\infty) \ar[r]  & F_\infty /\gamma_{c+1}(F_\infty) \ar[r] & F_\infty/\gamma_{c}(F_\infty) \ar[r] & 1 }\]
and where $\varphi'$ is the restriction of $\varphi_c$ to $\gamma_c(F_\infty)/\gamma_{c+1}(F_\infty)$. By induction, we have that 
$R(\varphi_{c-1})=1$. As the previous lemma shows that $R(\varphi')=1$, we find that $R(\varphi_c)=1$. 
\end{proof}

\begin{corollary}
Let $c>0$ be a positive integer and let $\varphi'_c$ be the automorphism of $M_{\infty,c}=F_\infty/(\gamma_{c+1}(F_\infty)F_{\infty}^{(2)})$ which is induced by $\varphi$.
Then $R(\varphi'_c)=1$. In particular, $M_{\infty,c}$, the free metabelian and nilpotent group of class $c$ on an infinite countable
number of generators does not have the $R_\infty$ property.
\end{corollary}

\section{Free  groups of infinite countable rank}

In the previous section we treated free nilpotent groups of infinite countable rank. In fact,
now we will generalize these results to all free groups of infinite countable rank. However, in this section the results will be less explicit then in the previous one.

On the other hand, in the appendix due to a result of R.\ Bianconi (\cite{Bi}) it is shown 
that Proposition~\ref{propuncountable}  of this section can even be generalized to free groups on any infinite set of generators.

\medskip

Consider any surjective map $\theta: \{x_0, x_1, x_2, x_3, \ldots \}\rightarrow F_\infty$ where 
$F_\infty$ is the free group on the $x_i$ in
such a way that $\theta(x_i)$ is a word involving only the variables 
$x_0,  x_1, \ldots, x_{i-1}$. So $\theta(x_0)=1, \ \theta(x_1)=x_0^k$ (for some $k$), etc. It is not difficult to see that such a map does exist.

\medskip


Now let $\varphi: F \longrightarrow F$ be the homomorphism determined by 
$\varphi(x_i)= \theta(x_i) x_i$.
Then $\varphi$ is an automorphism of $F_\infty$ (We leave it to the reader to define a map 
$\psi:\{x_0,x_1,\ldots\}\rightarrow F_\infty$ inductively, determining a morphism 
of $F_\infty$ which is the inverse of $\varphi$). 

Now, for any  $w$ of $F_\infty$, there exists an $x_i$ such that 
\[ w = \theta(x_i)= \varphi(x_i) x_i^{-1}.\]
Said differently for any $v=w^{-1}$ of $F_\infty$, we have that $v= w^{-1} = x_i \varphi(x_i)^{-1}$, for 
some $x_i$, which shows any $v\in F_\infty$ is Reidemeister equivalent to $1$.
We have shown:
\begin{proposition}\label{propuncountable}
With the notations of above, we have that $R(\varphi)=1$. Hence, $F_\infty$ does not have 
property $R_\infty$ and neither does any of the groups $N_{\infty,c}$, $M_{\infty,c}$ or
$S_{\infty, k}$.
\end{proposition}

\medskip

The construction  above can be generalized to obtain for any positive integer $n$
an automorphism  $\varphi_n: F_\infty  \to F_\infty$ which has Reidemeister number $n$.

\begin{theorem} Given a positive integer $n$  then there exists an automorphism 
$\varphi_n: F_\infty \to F_\infty$ such that the number of Reidemeister classes of $\varphi_n$ is $n$.
 \end{theorem}

\begin{proof}
Let us consider a map $\theta: \{x_0, x_1, x_2, x_3, ...\} \to F_\infty$ as above, i.e.\ it is surjective and $\theta(x_i)$ is a word involving only the variables 
$x_0,  x_1, ..., x_{i-1}$. 
Now, for $i=0,1,2,\ldots, n-1$ we let $J_i\subset \{x_0,x_1,x_2,\ldots\}$ 
be the subset of those  elements $x_k$ such that the exponent sum of the word $\theta(x_k)$ is congruent 
to $i$ mod $n$, and denote $W_i=\theta(J_i)$.   So we obtain that $ \{x_0,x_1,x_2,\ldots\}=J_0\cup J_1\cup\cdots \cup J_{n-1}$ and $ F_\infty=W_0\cup W_1\cup\cdots \cup W_{n-1}$ (disjoint unions). 

\medskip

Define the morphism $\varphi_n:F_\infty  \to F_\infty$ with $\varphi_n(x_k)=
\theta(x_k)^{-1}x_kx_0^{i}$, where $i\in \{0,1,\ldots,n-1\}$ is such that $x_k\in J_i$. 
Because of the properties of $\theta$ it is not hard to show that $\varphi_n$ is an
automorphism of $F_\infty$.


Now we claim: 
\begin{enumerate}
\item\label{claim1} Any element $w\in F$ belongs to the Reidemeister class of one of the elements of the set $\{1,x_0,x_0^2,...,x_0^{n-1}\}$.
\item\label{claim2} The elements $\{1,x_0,x_0^2,...,x_0^{n-1}\}$ belong to distinct Reidemeister classes.
\end{enumerate}
Let us proof claim (\ref{claim1}). 
Given $w\in F_\infty$ then there exist $x_{\ell}$ such that $\theta(x_{\ell})=w$
where  $x_{\ell}\in J_k$ for some $k\in \{0,...,n-1\}$. Then $x_0^k$ belongs to the same Reidemeister class as the element  $x_{\ell}x_0^k\varphi_n(x_{\ell})^{-1}=
x_{\ell}x_0^k(\theta(x_{\ell})^{-1}x_{\ell}x_0^k)^{-1}=
x_{\ell}x_0^k x_0^{-k}x_{\ell}^{-1}\theta(x_{\ell})=w$ and  the result folllows.

\medskip

To show   part (\ref{claim2}) observe that if two elements $w_1, w_2$ are in the same Reidemeister class, then
 they have the same exponent sum modulo $n$. This follows because for any $w\in F_\infty$ and any 
 generator $x_k\in J_i$ 
 (for some $i$) the related elements $w$ and  
 $x_k w(\varphi_n(x_k))^{-1}=x_kw(\theta(x_k)^{-1}x_kx_0^{i})^{-1}$ have the same exponent sum  from the defintion of $J_i$.  So the proof is complete.

\end{proof}


\section*{Appendix}

A relevant ingredient to construct an automorphism  of $F_\infty$ with Reidemeister number 1 was the existence of a surjective map $\theta: \{x_0, x_1, x_2, x_3, ..\}\rightarrow F_\infty$ where 
$F_\infty$ is the free group on the $x_i$ in
such a way that $\theta(x_i)$ is a word involving only the variables 
$x_0,  x_1, ..., x_{i-1}$. This result has been generalized by R. Bianconi as follows:

\begin{theorem} Let $X$ be an infinite set and $F(X)$ the free group generated by $X$. Let $\kappa=|X|$ be the cardinality of $X$ and enumerate $X=\{x_{\alpha}:\alpha<\kappa\}$. Define $X_{\alpha}=\{x_{\beta}:\beta<\alpha\}$. Then there is a surjection $h:X\to F(X)$, such that for all $\alpha<\kappa$, $h(x_{\alpha})\in F(X_{\beta})$, for some $\beta<\alpha$, assuming that $F(\varnothing)=\{1\}$.

\end{theorem}

\begin{proof}
For convenience, we treat firstly the case $|X|=\omega=\aleph_0$.

Enumerate $F(X)=\{y_n:n<\omega\}$. For each $k<\omega$ let $n_k=\min\{m<\omega: y_k\in F(X_m)\}$. Now we inductively choose an increasing sequence $p_k>n_k$, $k<\omega$. Define $h:X\to F(X)$ by $h(x_{p_k})=y_k$, and $h(x_j)=1$, where $j<\omega$ are the indices not present in the sequence $p_k$. In this way we have proved the theorem for the denumerable case.

\medskip

For the uncountable case we enumerate $F(X)$ conveniently in order to deal with both cases of regular and of singular cardinalities. We use $\aleph_{\beta}$ for the cardinal and $\omega_{\beta}$ for the corresponding (von Neumann) ordinal.

Let $|X|=\aleph_{\gamma}$, with $\gamma>0$. Enumerate $F(X_{\omega_{\alpha+1}})\setminus F(X_{\omega_{\alpha}})=\{y_{\beta}:\omega_{\alpha}\leq\beta<\omega_{\alpha+1}\}$, for each $\alpha$ such that $0\leq\alpha<\gamma$ Observe that for each limit ordinal $\lambda<\gamma$, $F(X_{\lambda})=\bigcup_{\beta<\lambda}F(X_{\beta})$, so that enumeration carries directly over to the limit cardinal case.

Let $h_{0}:X_{\omega}\to F(X_{\omega})$ be the surjection just obtained.

Suppose that we have already defined $h_{\omega_{\beta}}:X_{\omega_{\beta}}\to F(X_{\omega_{\beta}})$ satisfying the condition of the theorem, for $\beta<\gamma$.

if $\gamma$ is a limit ordinal, then we define $h_{\omega_{\gamma}}:X_{\omega_{\gamma}}\to F(X_{\omega_{\gamma}})$ as the union of all the $h_{\omega_{\beta}}$ and we are done, because $\omega_{\gamma}=\bigcup_{\beta<\gamma}\omega_{\beta}$.

Otherwise, $\gamma=\beta+1$ and we argue as follows.

For each ordinal $\eta$, $\omega_{\beta}\leq\omega_{\beta+1}$, let $\varphi_{\eta}=\min\{\xi:y_{\eta}\in F(X_{\xi})\}$. Because of the enumeration we have done in $F(X)$, we have that $\omega_{\beta}<\varphi_{\eta}<\omega_{\beta+1}$. By transfinite induction we can obtain an increasing sequence of ordinals $\psi_{\eta}>\varphi_{\eta}$, such that $\omega_{\beta}<\psi_{\eta}<\omega_{\beta+1}$, by choosing the smallest ordinal bigger than the previously chosen and bigger than the corresponding $\varphi_{\eta}$. Because $\aleph_{\beta+1}$ is a regular cardinal, this choice can be done for all $\eta<\omega_{\beta+1}$ (the regularity of the cardinal guarantees that there are plenty of ordinals to be chosen at limit ordinals $\eta$).

We define $h_{\beta+1}$ by sending $x_{\psi_{\eta}}$ to $y_{\eta}$ and the remaining elements of $X_{\beta+1}\setminus X_{\beta}$ to the unit element 1.

This finishes the proof.
\end{proof}

We can now show that a free group on any infinite set of generators does not have property $R_\infty$.

\begin{proposition}
Let $X$ be an infinite set and $F(X)$ the free group generated by $X$. Let
$\varphi$ be the automorphism of $F(X)$ determined by 
 $\varphi(x_{\alpha})=h(x_{\alpha})x_\alpha$, where $h$ is the surjection from the theorem above. Then,
the homomorphism $\varphi$ is an automorphism. Further, 
we have that $R(\varphi)=1$. Hence, $F(X)$ does not have 
property $R_\infty$.
 \end{proposition}

\begin{proof} For the first part let us define inductively a map $\psi: X \to F(X)$.
Define $\psi(x_0)=x_0$. Suppose  inductively that given a $\beta<\kappa$, with $\kappa=|X|$,
 we have defined $\psi$ 
on $X_{\beta}$ such that $\psi\circ \varphi(x_{\alpha})=x_{\alpha}$.
Define $\psi(x_\beta)=\psi(h(x_{\beta}))^{-1}x_{\beta}$. Then we certainly have $\psi\circ\varphi(x_{\beta})=x_{\beta}$. 

It is obvious that the morphism $\varphi$ is surjective.
Since the compostion $\psi\circ\varphi$ is the identity,
$\varphi$ is injective too and hence $\varphi$ is an automorphism.

 That   $R(\varphi)=1$ is similar to the countable case. 
\end{proof}

\end{document}